\documentclass[a4paper,12pt,oneside,reqno]{amsart}
\usepackage{amsmath,amsthm,amsfonts,amssymb,ifpdf}

\usepackage{amssymb}
\usepackage{amsmath}
\usepackage{amsthm}
\usepackage{graphicx}
\usepackage[all]{xy}
\usepackage{enumerate}
\usepackage{tikz-cd}
\usepackage{subcaption}

\ifpdf
  \usepackage[
    pdftex,
    colorlinks,%
    linkcolor=blue,citecolor=red,urlcolor=blue,
    hyperindex,%
    plainpages=false,%
    bookmarksopen,%
    bookmarksnumbered%
  ]{hyperref}
 \usepackage{thumbpdf}
\else
  \usepackage{hyperref}
\fi
\newtheorem{thm}{Theorem}[section]
\newtheorem{lemma}{Lemma}[section]

\newtheorem{prop}{Proposition}[section]

\newtheorem{defn}{Definition}[section]
\newtheorem{corr}{Corollary}[section]
\newtheorem{exam}{Example}[section]
\newtheorem{rmrk}{Remark}[section]

\newtheorem*{Proof*}{Proof}
\begin{document}
\baselineskip=16pt

\setcounter{tocdepth}{1}

\providecommand{\keywords}[1]
{
  \small
  \textbf{\textit{Keywords---}} #1
}


\title[]{Interval maps where every point is eventually fixed}

\author[V. Kannan]{V. Kannan}
\address{SRM University - AP, Amaravati-522502, India.}
\email{kannan.v@srmap.edu.in}

\author[Pabitra Narayan Mandal]{Pabitra Narayan Mandal}
\address{ School of Mathematics and Statistics, University of Hyderabad, Hyderabad 500046, India.}
\email{pabitranarayanm@gmail.com}


\begin{abstract}
Among the orbit patterns that force only eventually fixed trajectories, we completely describe the forcing relation, by answering the question: which orbit patterns force which others?
\end{abstract}
\maketitle
\section{Introduction}
\subsection{History}
This paper is a sequel to \cite{shark}, \cite{bald}, \cite{misi} where the forcing relation among orbit-patterns is investigated. In \cite{shark}, it is done for cycle-lengths. In \cite{bald}, it is done for cyclic patterns. In \cite{misi}, it is done for some combinatorial patterns. In this paper, the same is done for eventually fixed patterns. 

This paper is also a sequel to \cite{sharkov}, \cite{kannan}, and \cite{pillai} where the orbit-patterns of some simple dynamical systems have been investigated. In those papers, the classes of interval maps studied were

i) those with the zero entropy

ii) those with only finitely many types of orbits

iii) those with only three or four non-ordinary points (and therefore with only 3 or 4 types of orbits).

Along similar lines, the class studied in this paper is $$\mathcal{EF}:=\{\text{Interval maps where every element is eventually fixed}\}$$ We describe how their orbit patterns are better understood through a formal language. There are uncountably many conjugacy classes but only countably many orbit-patterns for them.

This paper is also a sequel to \cite{devaney}, \cite{kannan1} and \cite{kannan2} where the formal languages increase the convenience in the study of dynamics of interval maps. In one of them \cite{devaney}, it is the language of itineraries of a unimodal map (an important particular case of piece wise monotone maps (for details see \cite{milnor})). In another  \cite{kannan1} it is the language of locating periodic points of period 1 and 2. In yet another \cite{kannan2}, it is an index set for the set of orbit-patterns.

\subsection{Statement of Main Result}
There is a natural bijection between the language $\{L,R\}^*$ and the set of all $\mathcal{EF}$ orbit patterns. The forcing relation on the latter set receives a neat description when framed in the terminology of theory of languages. We are able to find four rules of derivation in $\{L,R\}^*$ so that the following theorem becomes true: An orbit pattern $\alpha$ forces another orbit pattern $\beta$ if and only if the corresponding word of $\beta$ can be derived from that of $\alpha$ using four rules of derivation (described below). 
 
This result is more charming because these rules of derivation are not the ones that an expert could have guessed in the beginning. 

The four rules of derivation are

\begin{itemize}
\item Reduction of $LL$ to $L$
\item Reduction of $RR$ to $R$
\item Reduction of $LR$ or $RL$ to the empty word
\item Formation of a tail.
\end{itemize}

\subsection{Preliminaries}
An interval map means a continuous self map $f$ on an interval $I$. The $f$-trajectory of a point $x\in I$ is the sequence $(f^n(x))_{n=0}^{\infty}$, where $f^n = f \circ f \circ f \circ ... \circ f$ ($n$ times) and $f^0(x)=x$. This is also called as an $f$-orbit of $x$. We also call $f$-trajectory or $f$-orbit of $x$ as merely trajectory or orbit of $x$ when the function is understood. Strictly speaking, in the literature $I$ is commonly referred to as $[0,1]$. But one can always transfer the system to an interval where $0$ is an interior point with out changing the dynamics of the system.

Let $f$ be an interval map and $x_1, x_2,...,x_{n+1}$ be such that $f(x_j)=x_{j+1}$ where $j=1,2,$...$,n$ and $f(x_{n+1})=x_{n+1}$.  Label every term of $x_n$ with $L$ or $R$ according as it moves to its left or right. This labelling is done only upto the $n$-th term because $x_{n+1}$ is a fixed point. An empty word corresponds to the orbit of a fixed point. The word $w$ corresponding to $f$-orbit of $x$ is also called as orbit pattern tag of $x$.

\begin{defn}(Orbit Pattern)
Two real sequences $(a_n)_{n=0}^\infty$ and $(b_n)_{n=0}^\infty$ are said to be of the same order-pattern if $a_m < a_n$
$\iff$ $b_m < b_n$ holds for all $m, n\in \mathbb{N}_0$. An order-pattern of a sequence $(f^n(x))_{n=0}^\infty$ in a real
dynamical system $(\mathbb{R}, f)$ or $(I, f)$ is called an orbit-pattern.
\end{defn}  

The orbit-pattern provides the information:
For each pair $(m,n)$ of non-negative integers, which is smaller between the two numbers $f^m(x)$ and $f^n(x)$? Moreover it has been proved \cite{sharkov}, \cite{kannan2} that any $\mathcal{EF}$ orbit pattern is determined by the word representing it. In other words two different orbit patterns will never have the same tag.

If a word $w$ represents the orbit pattern of an element $x$, then the size of the orbit of $x$ $=$ length of the word $+$ $1$ $=$ the time taken by $x$ to reach the fixed point. Throughout this paper we use $u$, $u'$, $v$, $v'$, $w$, $w'$, $w''$ as words over $\{L,R\}$ and $|w|$ to denote the length of $w$.  

\begin{defn}(Forced Word)
We say a word $u$ is forced by $w$, if every continuous interval map that admits the orbit pattern tag $w$ has to admit the orbit pattern tag $u$ also. 
\end{defn}

Since any interval map always contains a fixed point, any word forces empty word. We use the notation ``$w \to u$" to mean that $w$ forces $u$ where $w$ and $u$ are two words over $\{L,R\}$.

\begin{defn}(Derived Word)
We say a word $u$ is derivable from $w$ if it is obtained by successive use of the following rules (finite number of times in any order):

$1$. $vRRv'$ $\implies$ $vRv'$.

$1'$. $vLLv'$ $\implies$ $vLv'$.

$2$. $vLRv'$ $\implies$ $vv'$.

$2'$. $vRLv'$ $\implies$ $vv'$.

$3$. $vv'$ $\implies$ $v'$.
\end{defn}

Rules $1$, $1'$, $2$, $2'$ are called as rules of reduction and the rule 3 is called tail formation. We use the notation ``$\implies$" for a process of derivation.

\begin{defn}(Constructed Word) \label{cons}
We say a word $u$ is constructed from $w$ if $u\in \mathcal{L}_w$ where $\mathcal{L}_w$ is the language corresponding to $w$ and it is defined recursively as follows: 

If $|w|\leq 2$, then $\mathcal{L}_w$ consists of only tails of $w$.

If $|w|>2$ then we have

\[
    \mathcal{L}_w := 
    \begin{cases}
          \mathcal{L}_{Rw'}\cup L(\mathcal{L}^R_{Rw'}) & \textit{when}\; w=LRw', \\
          \mathcal{L}_{LLw'}\cup L(\mathcal{L}^L_{LLw'}) & \textit{when}\; w=LLLw', \\
          \mathcal{L}_{LRw'}\cup L(\mathcal{L}_{LRw'}) & \textit{when}\; w=LLRw',
    \end{cases}
\]
or dually when $w$ starts with $R$. Here $\mathcal{L}^L_{w'}$ means the set of all those words in $\mathcal{L}_{w'}$ that start with $L$ and similarly for the superfix $R$. Moreover $L(\mathcal{L}_{w'})$ is the set of all words $Lu$ where  $u\in \mathcal{L}_{w'}$.
\end{defn}

We use ``$<a,b>$" for a closed interval with end points $a$ and $b$. For two intervals $I$ and $J$, we write ``$I\xrightarrow{f} J$" if $f(I)\supset J$. We use the abbreviation IVT for Intermediate Value Theorem.

In the next three sections we prove our main result by proving three implications of the following triangular diagram. 
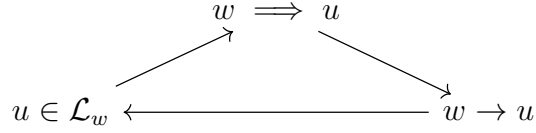
\[
\begin{tikzcd}
{} & w\implies u \arrow{dr} \\
 u\in \mathcal{L}_w \arrow{ur}  &&  \arrow{ll} w\to u
\end{tikzcd}
\]

\section{If $u$ is derivable from $w$ then $u$ is forced by $w$}

\begin{thm}\label{tail}
If $w$ is an orbit pattern tag in $\mathcal{EF}$, then

1. $w$ always forces its own tails. 

2. length of any orbit pattern forced by $w$ is $\leq |w|$.

\end{thm}

\begin{proof}
(Proof of 1): If $w$ is the orbit pattern tag of $x_1$, then the tail of length $|w|-j$ is the orbit pattern tag of $x_{j+1}$ where $1\leq j< n$.

(Proof of 2): Let $(x_j)$ be an $\mathcal{EF}$ orbit of length $n+1$ where $x_{n+1}$ is a fixed point. Therefore $x_i$ is on one side of $x_j$ for all $1\leq j<i\leq n+1$ (By \cite{sharkov}). Consider a map $f$ such that it is linear on $<x_i,x_j>$ whenever there is no other term in this interval. (In other words $(x_i,f(x_i))$ and $(x_j,f(x_j))$ are joined by a line segment in the graph of $f$.)  In this $f$, we have $f([m,M])\subset [m,M]$ where $m:=min\{x_1,x_2,...,x_n,x_{n+1}\}$ and $M:=max\{x_1,x_2,...,x_n,x_{n+1}\}$ and $f^n([m,M])=x_{n+1}$. Therefore the length of the forced orbit pattern by $w$ is at most $|w|$.

\end{proof}
  
\begin{rmrk}  
There are two more observations for $\mathcal{EF}$ orbit patterns which are important to note, namely

1. Only forced orbit pattern by $w$ of length $|w|$ is $w$ itself.

2. The forcing relation becomes a partial order on $\mathcal{EF}$ orbit patterns. However this result is not true for general orbit patterns. For instances, see \cite{misi}, \cite{kannan}.
\end{rmrk}


\begin{thm}\label{reduction}
Among $\mathcal{EF}$ orbit patterns,

$1$. $uRRw \to uRw$

$1'$. $uLLw \to uLw$

$2$. $uLRw \to uw$

$2'$. $uRLw \to uw$
\end{thm}

\begin{proof}
Proof of (1): Let $u$ be a word of length $k$. Let $(a_n)$ be an $f$-trajectory represented by $uRRw$. Assume that this orbit converges to $0$, then $a_{k+1}<a_{k+2}<0$. Then $f(0)>a_{k+2}$ and $f(a_k)<a_{k+2}$. By IVT, $\exists \; b_k\in <a_k,0>$ such that $f(b_k)=a_{k+2}$. If $a_k<0$ then $-b_k>0$ and $a_k-b_k<0$; if $a_k>0$ then $-b_k<0$ and $a_k-b_k>0$. Therefore by IVT, $\exists \; b_{k-1}\in <a_{k-1},0>$ such that $f(b_{k-1})=b_{k}$ where $b_{k-1}$ is of same sign as of $a_{k-1}$. In this way we get $b_j\in <a_j,0>$ where $j=1,2,...,k-1$ such that $f(b_j)=b_{j+1}$ and sign of $b_j$ is same as of $a_j$. Therefore the orbit pattern tag of $b_1$ is $uRw$.

Proof of (1$'$): Similar as (1).

\noindent \textbf{Claim-I}: $uRRLw \to uRw$ and $uLLRw \to uLw$.

Let $u$ be a word of length $k$. Let $(a_n)$ be an $f$-trajectory represented by $uRRLw$. Assume that this orbit converges to $0$, then $a_{k+1}<a_{k+2}<0<a_{k+3}$. Now $f(a_{k+1})<a_{k+4}$ and $f(a_{k+2})>a_{k+4}$. By IVT, $\exists \; b_{k+1}\in (a_{k+1},a_{k+2})$ such that $f(b_{k+1})=a_{k+4}$ and corresponding label of $b_{k+1}$ is $R$. Now $f(a_k)<b_{k+1}$ and $f(0)>b_{k+1}$. Therefore by IVT, $\exists \; b_k\in <a_k,0>$ such that $f(b_k)=b_{k+1}$. This shows that sign of $b_k$ and $a_k$ are same and hence the corresponding label. In this way we can find $b_j\in <a_j,0>$ where $j=1,2,...,k$ such that $f(b_j)=b_{j+1}$ where $j=1,2,...,k$ and $f(b_{k+1})=a_{k+4}$ and sign of $b_j$ is same as that of $a_j$. Therefore the orbit pattern of $b_1$ is $uRw$.

Similarly for $uLLRw \to uLw$.

\noindent \textbf{Claim-II}: $uLRLw \to uLw$ and $uRLRw \to uRw$.

Let $u$ be a word of length $k$. Let $(a_n)$ be an $f$-trajectory represented by $uLRLw$. Assume that this orbit converges to $0$. Then $a_{k+2}<0<a_{k+3}<a_{k+1}$. Now $f(a_k)>a_{k+3}$ and $f(0)<a_{k+3}$. By IVT, $\exists \; b_k\in <a_k,0>$ such that $f(b_k)=a_{k+3}$. By similar method as earlier we find $b_1$, $b_2$,..., $b_{k-1}$ such that $f(b_j)=b_{j+1}$ where $b_j\in <a_j,0>$, $j=1,2,...,k-1$. Therefore the orbit pattern tag of $b_1$ is $uLw$.

Similarly for $uRLRw \to uRw$.

Proof of (2): If $u$ is an empty word then $uLRw$ forces $uw$ as $uLRw=LRw$ and $uw=w$ and we know by Theorem \ref{tail} every tail is forced by their orbit always. If $w$ is an empty word, then $uLR$ is same as either $u'LLR$ or $u'RLR$ for some $u'$ where $u=u'L$ or $u'R$. Now $u'LLR$ forces $u'L$, i.e., $u$ (by Claim-I) and $u'RLR$ forces $u'R$, i.e., $u$ (by Claim-II). Hence if $w$ is empty then in either way $uLR$ forces $u$.

Assume that both $u$ and $w$ are non-empty. Then $uLRw$ is either $u'LLRw$ or $u'RLRw$ where $u=u'L$ or $u'R$. By Claim-I, $u'LLRw$ forces $u'Lw$, i.e., $uw$; by Claim-II, $u'RLRw$ forces $u'Rw$, i.e., $uw$. 

Proof of (2$'$): Similarly as (2).
\end{proof}

\begin{corr}
If $w$ is an $\mathcal{EF}$ orbit pattern tag and $u$ is derivable from $w$, then $u$ is forced by $w$.
\end{corr}
\begin{proof}
Forcing relation is transitive and every derived word is a combination of reductions and tail formation. By Theorem \ref{tail} and Theorem \ref{reduction}, every derived word is forced.
\end{proof}

\section{If $u$ is forced by $w$ then $u$ can be constructed for $w$}

Let $\mathcal{L}_w$ be the language constructed for $w$, now we describe $\mathcal{L}_{Lw}$ using the following rule in terms of $\mathcal{L}_w$:

1. If $w$ starts with $LL$, then $\mathcal{L}_{Lw}=\mathcal{L}_w\cup L(\mathcal{L}^L_w)$.

2. If $w$ starts with $LR$, then $\mathcal{L}_{Lw}=\mathcal{L}_w\cup L (\mathcal{L}_w)$.

3. If $w$ starts with $R$, then $\mathcal{L}_{Lw}=\mathcal{L}_w \cup L(\mathcal{L}^R_w)$. 

Similarly for $\mathcal{L}_{Rw}$.

\begin{thm}
If $u$ is forced by $w$ where $w$ be an $\mathcal{EF}$ orbit pattern, then $u\in \mathcal{L}_w$.
\end{thm}

\begin{proof}
We will prove this by using the strong principle of induction on the length of words.

Initial Step: If $|w|=0$, then nothing to prove. If $|w|=1$, then $w=L$ or $R$. We construct a corresponding function $f$ (according to $w=L$ or $R$) on $I$. Without loss of generality we may assume that $0$ is an interior point of $I$. 

When $w=L$,

\[
    f(x) = 
    \begin{cases}
         0 & \text{when} \; x>0\\
        x & \text{when} \; x\leq 0
    \end{cases}
\]

When $w=R$,
\[
    f(x) = 
    \begin{cases}
         0 & \text{when} \; x<0\\
        x & \text{when} \; x\geq 0
    \end{cases}
\]

It is easy to see that for this function there is no orbit of length more than 1. Therefore it does not force other than its tail. Hence the result is true for $|w|=1$.

Induction Hypothesis: Let us assume that the result is true for the words of length upto $n$ where $n\geq 2$. 

Inductive Step: Let $v$ be a word of length $n+1$. Then $v=Lw$ or $Rw$ where $w$ be a word of length $n$. Then $v=$ $LLv'$ or $RRv'$ or $LRv'$ or $RLv'$, where $v'$ be a word of length $n-2$. Let $v$ be the orbit pattern tag of $x_1$, where $f(x_j)=x_{j+1}$  for $1\leq j\leq n + 1$ and $f(x_{n+2})=x_{n+2}$. Without loss of generality we may assume $x_{n+2}=0$.

\noindent Case 1: When $x_1>0$ and $x_2>0$, i.e., $v=LLv'$. 

If $x_3>0$, then $f([x_2,x_1])\supset [x_3,x_2]$ for any in interval map $f$. The additional orbit patterns available in this case are $L$ concatenated with the orbit pattern available in $[x_3,x_2]$ and nothing more because $f$ can be taken as linear on $[x_2,x_1]$ (so that $f([x_2,x_1])=[x_3,x_2]$). All the orbit patterns available in $[x_3,x_2]$ starts with $L$. So the possible additional orbit patterns available are $L$ concatenated with the orbit patterns (constructed for $Lv'$) starting with $L$. In this case, $\mathcal{L}_v=\mathcal{L}_{LLv'}=\mathcal{L}_{Lv'}\cup L(\mathcal{L}^L_{Lv'})$.

If $x_3<0$, then $f([x_2,x_1])\supset [x_3,x_2]$ for any interval map $f$. The additional orbit patterns available in this case are $L$ concatenated with the orbit pattern available in $Lv'$ and nothing more because $f$ can be taken as linear on $[x_2,x_1]$ (so that $f([x_2,x_1])=[x_3,x_2]$). Therefore $\mathcal{L}_v=\mathcal{L}_{LLv'}=\mathcal{L}_{Lv'}\cup L(\mathcal{L}_{Lv'})$.

\noindent Case 2:  When $x_1<0$ and $x_2<0$, i.e., $v=RRv'$. Similarly for Case 1.



\noindent Case 3: When $x_1>0$ and $x_2<0$, i.e., $v=LRv'$.

Take $x_M= max \{x_j:\: 2\leq j \leq n+2\}$. Now either $f(x_M)=0$ or $f(x_M)>0$ or $f(x_M)<0$.

If $f(x_M)<0$, then $f([x_M,x_1])\supset [x_2,x_{M+1}]$ for any interval map $f$. Therefore the additional orbit patterns available in this case are $L$ concatenated with the orbit patterns available in $[x_2,x_{M+1}]$ and nothing more because $f$ can be taken as linear on $[x_M,x_1]$ (so that $f([x_M,x_1])=[x_2,x_{M+1}]$). All the orbit patterns available in $[x_2,x_{M+1}]$ start with $R$. Therefore the possible additional orbit patterns available are $L$ connected with the orbit patterns (constructed for $Rv'$) starting with $R$. In this case, $\mathcal{L}_v=\mathcal{L}_{LRv'}=\mathcal{L}_{Rv'}\cup L(\mathcal{L}^R_{Rv'})$.

If $f(x_M)=0$, then $f(x_M,x_1)\supset [x_2,0]$ for any interval map $f$. Therefore the additional orbit patterns available in this case are $L$ concatenated with orbit patterns available in $[x_2,0]$ and nothing more because $f$ can be taken as linear on $[x_M,x_1]$ (so that $f([x_M,x_1])=[x_2,0]$). In this case also, $\mathcal{L}_v=\mathcal{L}_{LRv'}=\mathcal{L}_{Rv'}\cup L(\mathcal{L}^R_{Rv'})$.

If $f(x_M)>0$, then $f([x_{M}, x_1])\supset [x_2,x_{M+1}]$ for any interval map $f$. Therefore the additional orbit patterns available in this case are $L$ concatenated with the orbit patterns (constructed for $Rv'$) starting with $R$ and nothing more because $f$ can be taken as linear on $[x_M,x_1]$ (so that $f([x_M,x_1])=[x_2,x_{M+1}]$). In this case also, $\mathcal{L}_v=\mathcal{L}_{LRv'}=\mathcal{L}_{Rv'}\cup L(\mathcal{L}^R_{Rv'})$.

\noindent Case 4: When $x_1<0$ and $x_2>0$ i.e., $v=RLv'$. Similarly for Case 3.





Hence the result.

\end{proof}

\section{If $u$ is constructed for $w$ then $u$ is derivable from $w$}
It is important to observe that one step reduction and one step tail formation always commute. Therefore every derived word from $w$ is a tail of a reduced word of $w$. Moreover if $u$ is derived from $w$, then there exists a derivation process in which tail formation is not there except in the last step. 

\begin{lemma}\label{inv}
If $w'$ is derived from $w$ by one step reduction only (without tail formation). Then we have:

i) If $w$ starts with $L$, then $w'$ also starts with $L$.

ii) If $w$ starts with $R$, then $w'$ also starts with $R$. 
\end{lemma}

\begin{proof}
In the first case, let us assume $w=Lw''$ for some $w''$.

\noindent Case-I: If there is a reduction on $w''$, then $w'$ starts with $L$.

\noindent Case-II: If there is a reduction on first two letters of $w$, then the derived word is also obtained by tail formation also. This is not allowable by hypothesis. Hence Case-II will not arise for one step reduction only.

Hence the result.

\end{proof}

\begin{rmrk}\label{process}
By Lemma \ref{inv}, we get that one can keep the first letter throughout the reduction process i.e., till the previous step to tail formation. 
\end{rmrk}

\begin{lemma}\label{caseL}
If $u$ is a tail of $w$, where $w$ starts with $L$, then $Lu$ is reduced from $Lw$. And dually if $u$ is a tail of $w$, where $w$ starts with $R$, then $Ru$ is reduced from $Rw$ 
\end{lemma}
\begin{proof}
Say $w=Lw'$ for some $w'$.

\noindent Case I: If $u=w$, then nothing to prove.

\noindent Case II: If $u\neq w$, then $w=Lw''u$ for some $w''$. Applying repeatedly four rules of reduction on $Lw''u$, we may get $LRu$ or $Lu$. Now applying once again the rule of reduction on $LLRu$ or on $LLu$ we get $Lu$. 

Therefore, if $u$ is a tail of $w$ and $w$ starts with $L$, then $Lu$ can be reduced from $Lw$.
\end{proof}

\begin{lemma}\label{caseR}
If $Ru$ is a tail of $w$ and $w$ starts with $R$, then $LRu$ is reduced from $Lw$.
\end{lemma}

\begin{proof}
If $w=Ru$, then nothing to prove as $LRu=Lw$ is a trivial reduction on $Lw$.

If $w\neq Ru$, then $w=Rw'Ru$ for some $w'$. Repeatedly applying four rules of reductions on $Rw'Ru$, we may get two possible words, namely $Ru$ or $RLRu$. Hence $LRu$ can be reduced either from $LRu$ by trivial reduction or from $LRLRu$ by one more reduction. Therefore $LRu$ is reduced from $Lw$.   
\end{proof}

\begin{prop}\label{propo}
If $Lu$ is derivable from $w$, where $w$ starts with $L$, then $LLu$ is derivable from $Lw$. And dually if $Ru$ is derivable from $w$, where $w$ starts with $R$, then $RRu$ is derivable from $Rw$.
\end{prop}

\begin{proof}
Since $w\implies Lu$, there exists a derivation process such that possibly except the last step all the intermediate steps are reduction only.

\noindent Case I: If $w\implies Lu$, where all the intermediate steps are reduction, then $Lw\implies LLu$.

\noindent Case II: If not the Case I, then $\exists \; w'$ such that $w\implies w'$ (reduction steps) and $w' \implies Lu$ (tail formation step). By Remark \ref{process}, we may assume that $w'$ starts with $L$. Therefore $Lu$ is a tail of $w'$, where $w'$ starts with $L$. By Lemma \ref{caseL}, $LLu$ is reduced from $Lw'$ and $Lw'$ is reduced from $Lw$. Therefore $LLu$ is reduced and hence derived from $Lw$.
\end{proof}

\begin{thm}
If $u\in \mathcal{L}_w$, then $u$ can be derived from $w$ i.e., $u$ is a tail of a reduced word of $w$.
\end{thm}
\begin{proof}
Without loss of generality, we may assume that $w$ starts with $L$. Say $w=Lw'$ for some $w'$. Assume that the result is true for all the words of length $< |w|$. By Definition \ref{cons}, we know 

1. If $w'$ starts with $LL$, then $\mathcal{L}_{Lw'}=\mathcal{L}_{w'}\cup L(\mathcal{L}^L_{w'})$.

2. If $w'$ starts with $LR$, then $\mathcal{L}_{Lw'}=\mathcal{L}_{w'}\cup L (\mathcal{L}_{w'})$.

3. If $w'$ starts with $R$, then $\mathcal{L}_{Lw'}=\mathcal{L}_{w'} \cup L(\mathcal{L}^R_{w'})$.

\noindent Case I: When $w'$ starts with $LL$. If $u\in \mathcal{L}_{w'}$, then by the induction hypothesis the result is true ($u$ is derivable from $w'$; and $w'$ is a tail of $w$. Then $u$ is derivable from $w$).

If $u\in L(\mathcal{L}^L_{w'})$, then $u=LLu'$ where $Lu'$ is derivable from $w'$. Since $w'$ starts with $L$, by Proposition \ref{propo}, $LLu'$ is derivable from $Lw'=w$.

\noindent Case II: When $w'$ starts with $LR$. If $u\in \mathcal{L}_{w'}$, then $u$ is derivable from $w$ (by same argument as above).

If $u\in L(\mathcal{L}_{w'})$, then $u=Lu'$ where $u'$ is derivable from $w'$. By Remark \ref{inv}, we may assume $\exists \; w''$ which starts with $L$ such that $w\implies w''$ (reduction steps) and $w'' \implies u'$ (tail formation step). By Lemma \ref{caseL}, $Lu'$ is reduced from $Lw''$. Observe that $Lw''$ is reduced from $Lw'=w$ also. Hence $u$ is derivable from $w$.

\noindent Case III: When $w'$ starts with $R$. If $u\in \mathcal{L}_{w'}$, then $u$ is derivable from $w$ (by same argument as above).

If $u\in L(\mathcal{L}^R_{w'})$, then $u=LRu'$ where $Ru'$ is derivable from $w'$ starting with $R$. Now by Remark \ref{process}, $\exists \; w''$ where $w''$ starts with $R$ such that  $w'\implies w''$ (reduction steps) and $w'' \implies Ru'$ (tail formation step).
Therefore $Ru'$ is a tail of $w''$ where $w''$ starts with $R$. By Lemma \ref{caseR}, $u=LRu'$ is derivable from $w$.  

\end{proof}

\section{An Illustration and Drawing}
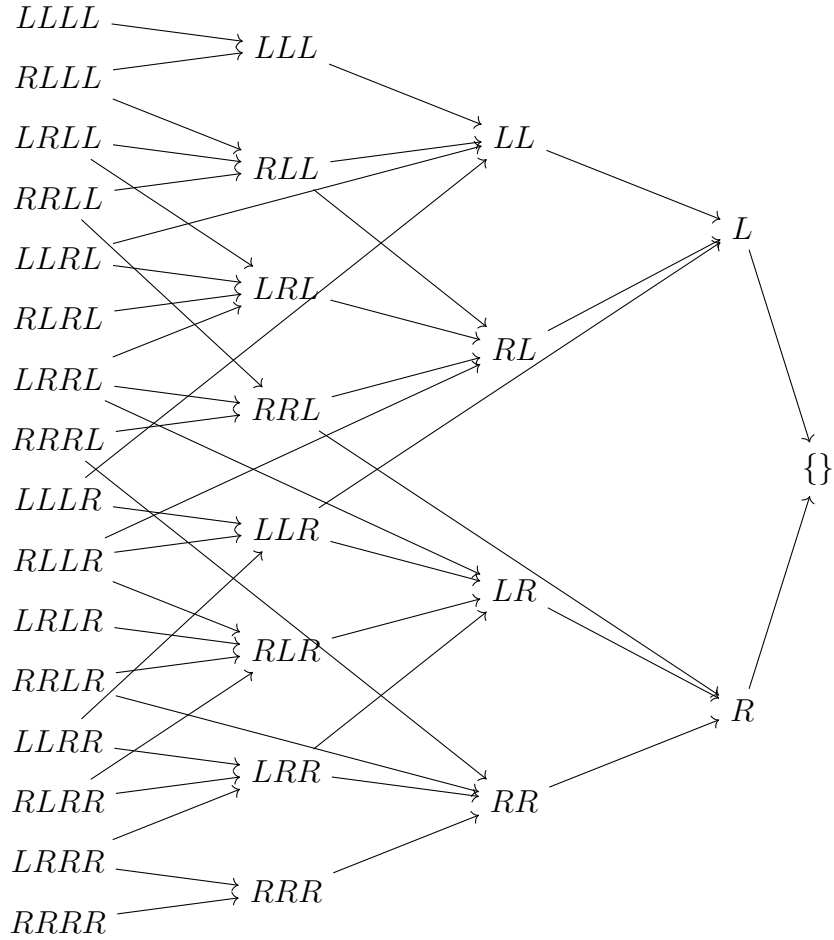
\begin{figure}[h]
	{\centering
		\begin{tikzpicture}
		\node (LLLL) at (0,11.8) {$LLLL$};
		\node (RLLL) at (0,11) {$RLLL$};
		\node (LRLL) at (0,10.2) {$LRLL$};
		\node (RRLL) at (0,9.4) {$RRLL$};
		\node (LLRL) at (0,8.6) {$LLRL$};
		\node (RLRL) at (0,7.8) {$RLRL$};
		\node (LRRL) at (0,7) {$LRRL$};
		\node (RRRL) at (0,6.2) {$RRRL$};
		\node (LLLR) at (0,5.4) {$LLLR$};
		\node (RLLR) at (0,4.6) {$RLLR$};
		\node (LRLR) at (0,3.8) {$LRLR$};
		\node (RRLR) at (0,3) {$RRLR$};
		\node (LLRR) at (0,2.2) {$LLRR$};
		\node (RLRR) at (0,1.4) {$RLRR$};
		\node (LRRR) at (0,0.6) {$LRRR$};
		\node (RRRR) at (0,-0.2) {$RRRR$};
		\node (LLL) at (3,11.4) {$LLL$};
		\node (RLL) at (3,9.8) {$RLL$};
		\node (LRL) at (3,8.2) {$LRL$};
		\node (RRL) at (3,6.6) {$RRL$};
		\node (LLR) at (3,5) {$LLR$};
		\node (RLR) at (3,3.4) {$RLR$};
		\node (LRR) at (3,1.8) {$LRR$};
		\node (RRR) at (3,0.2) {$RRR$};
		\node (LL) at (6,10.2) {$LL$};
		\node (RL) at (6,7.4) {$RL$};
		\node (LR) at (6,4.2) {$LR$};
		\node (RR) at (6,1.4) {$RR$};
		\node (L) at (9,9) {$L$};
		\node (R) at (9,2.6) {$R$};
		
		\node (E) at (10, 5.8) {$\{\}$};
		\draw[->] (L) -- (E);
		\draw[->] (R) -- (E);
		\draw[->] (LLLL) -- (LLL);
		\draw[->] (RLLL) -- (LLL);
		\draw[->] (RLLL) -- (RLL);
		\draw[->] (LRLL) -- (RLL);
		\draw[->] (LRLL) -- (LRL);
		\draw[->] (RRLL) -- (RLL);
		\draw[->] (RRLL) -- (RRL);
		\draw[->] (LLRL) -- (LL);
		\draw[->] (LLRL) -- (LRL);
		\draw[->] (RLRL) -- (LRL);
		\draw[->] (LRRL) -- (LRL);
		\draw[->] (LRRL) -- (RRL);
		\draw[->] (LRRL) -- (LR);
		\draw[->] (RRRL) -- (RRL);
		\draw[->] (RRRL) -- (RR);
		\draw[->] (LLLR) -- (LL);
		\draw[->] (LLLR) -- (LLR);
		\draw[->] (RLLR) -- (RL);
		\draw[->] (RLLR) -- (LLR);
		\draw[->] (RLLR) -- (RLR);
		\draw[->] (LRLR) -- (RLR);
		\draw[->] (RRLR) -- (RLR);
		\draw[->] (RRLR) -- (RR);
		\draw[->] (LLRR) -- (LLR);
		\draw[->] (LLRR) -- (LRR);
		\draw[->] (RLRR) -- (RLR);
		\draw[->] (RLRR) -- (LRR);
		\draw[->] (LRRR) -- (LRR);
		\draw[->] (LRRR) -- (RRR);
		\draw[->] (RRRR) -- (RRR);
		\draw[->] (LLL) -- (LL);
		\draw[->] (RLL) -- (LL);
		\draw[->] (RLL) -- (RL);
		\draw[->] (LRL) -- (RL);
		\draw[->] (RRL) -- (RL);
		\draw[->] (RRL) -- (R);
		\draw[->] (LLR) -- (L);
		\draw[->] (LLR) -- (LR);
		\draw[->] (RLR) -- (LR);
		\draw[->] (LRR) -- (LR);
		\draw[->] (LRR) -- (RR);
		\draw[->] (RRR) -- (RR);
		\draw[->] (LL) -- (L);
		\draw[->] (RL) -- (L);
		\draw[->] (LR) -- (R);
		\draw[->] (RR) -- (R);
		\end{tikzpicture}
		\caption{Hasse Diagram upto words of length 4}\label{Fig1}}
\end{figure}

\begin{exam}
Let $w=RLLRL$. Consider the corresponding piece wise linear map $f$ such that $-1 \xrightarrow{f} \frac{1}{2} \xrightarrow{f} \frac{1}{3} \xrightarrow{f} -\frac{1}{4} \xrightarrow{f} \frac{1}{5} \xrightarrow{f} 0$ where $0$ is a fixed point. Now we have

\[
    f(x) = 
    \begin{cases}
         \frac{1}{10}- \frac{2}{5} x & \; x\in [-1,-\frac{1}{4}],\\
        - \frac{4}{5} x & \; x\in [-\frac{1}{4},0],\\
        0 & \; x\in [0,\frac{1}{5}],\\
        \frac{3}{8} - \frac{15}{8} x & \; x\in [\frac{1}{5}, \frac{1}{3}],\\
        \frac{7}{2} x -\frac{17}{12} & \; x\in [\frac{1}{3}, \frac{1}{2}].
        
    \end{cases}
\]
\end{exam}

$[-1,-\frac{19}{21}) \xrightarrow{f} (\frac{97}{210}, \frac{1}{2}]\xrightarrow{f} (\frac{1}{5},\frac{1}{3}] \xrightarrow{f} [-\frac{1}{4},0) \xrightarrow{f} (0,\frac{1}{5}] \xrightarrow{f} 0.$

Here $L$, $RL$, $LRL$, $LLRL$, $RLLRL$ are available.

$[-\frac{19}{21}, -\frac{16}{21})\xrightarrow{f} (\frac{17}{42}, \frac{97}{210}]\xrightarrow{f} (0,\frac{1}{5}]\xrightarrow{f} 0.$

Here $L$, $LL$, $RLL$ are available

$[-\frac{16}{21},-\frac{1}{4})\xrightarrow{f} (\frac{1}{5}, \frac{17}{42}]\xrightarrow{f} [-\frac{1}{4},0]\xrightarrow{f} [0,\frac{1}{5}]\xrightarrow{f} 0.$

Here $L$, $RL$, $LRL$, $RLRL$ are available. Moreover, any interval map always contains (other than the empty word) the orbit patterns corresponding to $L$, $LL$, $RL$, $LRL$, $RLL$, $LLRL$, $RLRL$, $RLLRL$, whenever there is an orbit pattern tag $RLLRL$. And nothing more for this $f$. Hence the set of all forced orbit patterns (other than the fixed point) is $\{L,\: LL,\: RL,\: LRL,\: RLL,\: LLRL,\: RLRL,\: RLLRL\}$.

Now we will verify that same can be achieved by the set of construction rules. Using the construction rules, we have $\mathcal{L}_{RLLRL}=\mathcal{L}_{LLRL}\cup R(\mathcal{L}^L_{LLRL})$. Inductively, we have $\mathcal{L}_{LLRL}=\{LLRL,\: LRL,\: LL,\: L,\:RL\}$. So $\mathcal{L}_{RLLRL}=\{LLRL,\: LRL,\: RL,\: LL,\: L,\: RLLRL,\: RLRL,\: RLL\}$.

Again same can be achieved by using the set of derivation rules. Observe that $RLLRL \implies RLRL$ (using $LL\implies L$), $RLLRL\implies RLL$ (last $RL$ reduced to an empty word), $RLLRL \implies LLRL \implies LRL \implies RL \implies L$ and $RLL \implies LL$ (using tail formation). 


\noindent \textbf{Concluding Remarks}: The question we are dealing in this paper is very fundamental to understand the system. But it is not known for general orbit patterns. We can say for general orbit pattern, forced orbit patterns can be very complex. For instance, we take an eventually fixed orbit pattern $a<b<c<d$ such that $b\xrightarrow{f} d \xrightarrow{f} a \xrightarrow{f} c$. One can show that it is not an $\mathcal{EF}$ orbit pattern; it forces a $6$-cycle and hence it forces uncountably many orbit patterns.

\section*{Acknowledgement(s)}
The second author acknowledges NBHM-DAE (Government of India) for financial
support (Ref. No. 2/39(2)/2016/NBHM/R \& D-II/11397). This work was done during the second author's visit to SRM University-AP.

\end{document}